\documentclass[12pt]{amsart}
\usepackage{amsmath, amssymb, amsthm, latexsym}
\usepackage{amssymb,amscd,amsmath}
\usepackage{latexsym}
\usepackage[center]{caption}
\usepackage{tikz}
\newcounter{braid}
\newcounter{strands}

\DeclareMathAlphabet{\bsf}{OT1}{cmss}{bx}{n}

\pgfkeyssetvalue{/tikz/braid height}{1cm}
\pgfkeyssetvalue{/tikz/braid width}{1cm}
\pgfkeyssetvalue{/tikz/braid start}{(0,0)}
\pgfkeyssetvalue{/tikz/braid colour}{black}
\pgfkeys{/tikz/strands/.code={\setcounter{strands}{#1}}}

\makeatletter
\def\cross{%
  \@ifnextchar^{\message{Got sup}\cross@sup}{\cross@sub}}

\def\cross@sup^#1_#2{\render@cross{#2}{#1}}

\def\cross@sub_#1{\@ifnextchar^{\cross@@sub{#1}}{\render@cross{#1}{1}}}

\def\cross@@sub#1^#2{\render@cross{#1}{#2}}

\def\render@cross#1#2{
  \def\strand{#1}
  \def\crossing{#2}
  \pgfmathsetmacro{\cross@y}{-\value{braid}*\braid@h}
  \pgfmathtruncatemacro{\nextstrand}{#1+1}
  \foreach \thread in {1,...,\value{strands}}
  {
    \pgfmathsetmacro{\strand@x}{\thread * \braid@w}
    \ifnum\thread=\strand
    \pgfmathsetmacro{\over@x}{\strand * \braid@w + .5*(1 - \crossing) * \braid@w}
    \pgfmathsetmacro{\under@x}{\strand * \braid@w + .5*(1 + \crossing) * \braid@w}
    \draw[braid] \pgfkeysvalueof{/tikz/braid start} +(\under@x pt,\cross@y pt) to[out=-90,in=90] +(\over@x pt,\cross@y pt -\braid@h);
    \draw[braid] \pgfkeysvalueof{/tikz/braid start} +(\over@x pt,\cross@y pt) to[out=-90,in=90] +(\under@x pt,\cross@y pt -\braid@h);
    \else
    \ifnum\thread=\nextstrand
    \else
     \draw[braid] \pgfkeysvalueof{/tikz/braid start} ++(\strand@x pt,\cross@y pt) -- ++(0,-\braid@h);
    \fi
   \fi
  }
  \stepcounter{braid}
}

\tikzset{braid/.style={double=\pgfkeysvalueof{/tikz/braid colour},double distance=1pt,line width=2pt,white}}

\newcommand{\braid}[2][]{%
  \begingroup
  \pgfkeys{/tikz/strands=2}
  \tikzset{#1}
  \pgfkeysgetvalue{/tikz/braid width}{\braid@w}
  \pgfkeysgetvalue{/tikz/braid height}{\braid@h}
  \setcounter{braid}{0}
  \let\sigma=\cross
  #2
  \endgroup
}
\makeatother

\input xypic
\newtheorem{theorem}{Theorem}
\newtheorem{proposition}[theorem]{Proposition}

\newtheorem{lemma}[theorem]{Lemma}

\newtheorem{definition}[theorem]{Definition}


\def\Z{\mathbb{Z}}

\def\Pi{\mathbb{P}^{\infty}}

\def\qed{\hfill$\square$\medskip}

\def\Zpk{\mathbb{Z}/p^{k}}
\def\Zpk1{\mathbb{Z}/p^{k-1}}

\newcommand{\rref}[1]{(\ref{#1})}

\newcommand{\beg}[2]{\begin{equation}\label{#1}#2\end{equation}}
\def\r{\rightarrow}

\def\sl2{\widetilde{SL_{2}(\Z)}}
\def\smin{\smallsetminus}

\def\smin{\smallsetminus}

\title[$D$-structures]{$D$-structures and Derived Koszul duality for unital operad algebras}
\author{Tyler Foster, Po Hu and Igor Kriz}

\begin{document}

\begin{abstract}
Generalizing a concept of Lipshitz, Ozsv\'{a}th and Thurs-ton from Bordered Floer homology,
we define $D$-structures on algebras of unital operads, which
can also be interpreted as a generalization of 
a seemingly unrelated concept of Getzler and Jones. This construction gives rise
to an equivalence of derived categories, which can be thought of as a unital
version of Koszul duality using non-unital Quillen homology. We also discuss a multi-sorted 
version of the construction, which provides a framework for unifying the known algebraic
contexts of Koszul duality.
\end{abstract}

\maketitle

\section{Introduction}

Koszul duality for operads and algebras over them was introduced in the landmark paper
\cite{gk} by Ginzburg and Kapranov. For example, as long as we work over a field of
characteristic $0$, the operads encoding commutative and Lie algebras are Koszul dual,
while the operad encoding associative algebras is self-dual. The Koszul duality of Ginzburg
and Kapranov \cite{gk} has both a non-derived and a derived version. The derived
version involves a kind of a bar construction on an operad, which we call the
{\em Ginzburg-Kapranov bar construction}. The
Ginzburg-Kapranov bar construction on an operad is a differential graded co-operad,
which, under suitable finiteness assumptions, is the dual of a differential graded operad,
the {\em derived Koszul dual} of the original operad. When there is not enough finiteness
to dualize, the DG co-operad can also be used directly, and we refer to it as the
{\em Koszul transform} of the original operad. 
Just as in the more classical Koszul duality of associative algebras (cf. \cite{p}), there is
a property of an operad being {\em Koszul} which implies that the Koszul dual operad is, in fact,
quasi-isomorphic to a (non-differential graded) operad, which is then called the
{\em non-derived Koszul dual}.

The concept of derived Koszul duality is (as always) more important to homotopy theory foundations
than the corresponding non-derived concept, although the non-derived
concept is a useful calculational tool. Derived Koszul dual operads
are a part of a more general scheme, which in some sense goes
back to Quillen \cite{qc}: In any based category $C$ with finite
products, we may define the category of abelian objects $Ab(C)$. In good (and also typical)
cases, the forgetful
functor
$$U:Ab(C)\r C$$
has a left adjoint $L$, called {\em abelianization}. The functor $LU$ is then
a co-monad in $Ab(C)$. In the presence of a mechanism
creating derived categories, if we denote by $M$ the derived version of 
the co-monad $LU$, then
the Koszul transform of the derived category $DC$ should be the category of 
$M$-co-algebras. This is, of course, somewhat vague, and it gets even more so:
Abelian objects may be replaced by $E_\infty$-objects, and ultimately, the functor
$L$ may be replaced by Dwyer-Kan stabilization (an anague of the topological notion
of spectra in a based category with certain additional structure, cf. \cite{dk}).
It is very difficult to get any precise theorems along these lines in the most general
case because of convergence problems in the associated cobar spectral sequence.
Nevertheless, it was observed by Kontsevich that derived Koszul duality for operads is a part of this
general scenario (see \cite{kontsevich,konso, hkv, hstring}).

One of the puzzles of Koszul duality has been that it does not seem to interact well with units.
The main problem is that unital (non-augmented) operads, algebras over such operads etc. do not tend to
form a based category, and therefore the scheme described in the last paragraph does
not work directly. (When discussing a unital context, we mean, similarly as \cite{qc}, that there is a unit but not an 
augmentation. Similarly as for algebras, the category of augmented unital operads is canonically
equivalent to the category of non-unital operads by passing to the augmentation ideal.) 
Furthermore, if one takes Quillen cohomology of unital objects in the
corresponding non-unital category, one usually gets zero (\cite{qc}).
It has been an open question whether there is a version of Koszul duality which works,
say, for unital algebras over unital operads, in a sense which would generalize the above categorical
scheme. 

The main result of the present paper is to define derived Koszul
duality for unital algebras over unital operads. Our approach is to take non-unital homology of the unital 
$\mathcal{C}$-algebra, which by \cite{kontsevich,konso} is calculated by the
Ginzburg-Kapranov bar construction. Even though these bar constructions have $0$ chain homology,
they have however additional structure (which we call {\em $D$-structure}). There is a
natural way of defining morphisms and equivalences of $D$-structures, and prove that
the resulting derived category is equivalent to the category of $\mathcal{C}$-algebras. Comment: 
while we do rigorously construct the derived category of $D$-structures
for a general unital operad, we do not construct a Quillen model structure \cite{q} for them.
Constructing such a structure remains an interesting open problem. 

The term $D$-structure is taken from the 
{\em bordered Floer homology} of Lipshitz, Oszv\'{a}th and Thurston \cite{o},
which is a (pointed)
topological quantum field theory built out of the Heegaard-Floer homology invented by
Ozsv\'{a}th and Szab\'{o} \cite{os}.
From the point of view of the present paper, they use our derived Koszul duality in the
special case of $A$-modules where $A$ is a unital algebra: both the modules and their
Koszul transform $D$-structures occur naturally as combinatorial objects calculating Bordered Floer homology.
There are a few technical nuances, for example, on one side, \cite{o} consider $A_\infty$-modules
instead of strict modules, and also consider the differential as one of
the operations. We replace this by a context most suitable to our techniques; the different
variants of the concepts, both on the level of modules and $D$-structure, are easily seen to lead to equivalent
derived categories. In any case, it is very important to \cite{o} to have an
{\em equivalence of derived categories}; they need to construct a pairing of two objects in a certain geometrically
given category, while a natural geometrically given pairing is between an object of the category and
an object of its Koszul transform. We reproduce a full generalization of this equivalence of categories
to the general context of unital algebras over unital operads.

Interestingly enough, $D$-structures turn out to be also equivalent to the concept of quasi-free (or almost free) operad
algebras introduced, in the augmented case, in 1994 by Getzler and Jones in their unpublished preprint 
\cite{gjones2} (see also Fresse \cite{fresse}). We only need to notice that the concept makes sense in the 
non-augmented case. Of course,
it is 
the more explicit $D$-structure version of the concept which occurs naturally in 
\cite{o}, and is naturally interpreted as a version of Koszul duality. 
The equivalence of both concepts is an easy but very pretty observation, a clean and simple pattern 
in what has become known as a somewhat messy and technical area.

Since many papers were written on this subject to date, and the basic background definitions are 
not  unified (or even always completely rigorous), we chose to make the present paper self-contained, and
to deal with a number of technical issues. Ginzburg and Kapranov \cite{gk} work
over a field of characteristic $0$, while Lipshitz, Ozsv\'{a}th and Thurston \cite{o} work over
a field of characteristic $2$. The motivation of \cite{o} is to avoid the discussion of signs.
In this paper, we do work out the signs, and in a sense, this is one of our key points: 
While in \cite{gk}, a clever exterior algebra method is introduced to
control the signs in the Ginzburg-Kapranov bar construction, we need to extend this construction
by a {\em Clifford algebra} to handle the differential graded case. This becomes important because
the derived Koszul transform of a DG operad is, in fact, naturally {\em bigraded}. We need to totalize
this into a singly graded chain complex. 
While there are standard signs for totalizing double chain complexes,
these signs {\em do not work for generalizing the $D$-structures of Lipshitz, Ozsv\'{a}th and Thurston \cite{o}
outside of characteristic $2$}! The Clifford algebra method we use  reconciles the signs
of the internal grading with the Ginzburg-Kapranov bar construction completely, and works
in the setting we need.

To make a rigorous treatment of the signs, we introduce our own rigorous definition of ``planar trees". Many previous
papers rely only on pictorial intuition, which, in some cases, even obscures the crucial point of the 
role of identifications introduced by symmetric group actions on operads.
This is related to the more general point that while we are interested in derived
categories, $D$-structure is a point-set level algebraic structure, and therefore for our purpose, 
objects cannot be freely replaced by quasiisomorphic ones (this is even more apparent when we use the
interpretation of \cite{gjones2}). This is one of the reasons
why we must use the Ginzburg-Kapranov bar construction \cite{gk};
we do not know how to make our theory replacing it with a 
two-sided bar construction of monads \cite{may}, (which, of course, could
be used if we worked only
up to homotopy). In some sense, \cite{gk} deal with the same issue in their original concept of
non-unital Koszul duality. 

A related point is that the reason Ginzburg-Kapranov \cite{gk} work in characteristic $0$ is
that otherwise the monad associated with an operad does not preserve quasiisomorphisms.
This happens in some very basic cases, for example for commutative algebras.
This is a very well recognized phenomenon which plays an important role in homotopy theory (cf. \cite{mt},
or, in an algebraic context, \cite{km}). In this paper, we work in the category of modules over an
arbitrary field, and impose a condition we call $\Sigma$-cofibrancy on the operad. This condition
also comes from \cite{mt,km}, and has of course been since used by other authors,
too. It is automatically
satisfied in characteristic $0$.

Finally, we would like to comment briefly on the relationship of our results with previous work,
in particular the recent results
of Hirsh and Mill\`es \cite{hm}. The main concept of \cite{hm} is that of a {\em properad},
which is a variant of previous notions of PROPs \cite{val} and dioperads \cite{gan}. PROPs, like operads, were first
introduced in the context of infinite loop space theory \cite{bv}. PROP's (or dioperads) themselves
are algebras over appropriate {\em multisorted} operads. A derived version of the
concept of {\em curvature} of \cite{hm} is then equivalent to a special case of (a multisorted version of) D-structure.
The use of the term curvature in this context was, in fact, coined by Getzler and Jones \cite{gjones}.
That paper provided the first clue why a notion of curvature may be Koszul dual to a unit. In some
sense, \cite{hm} generalize that approach.
While explicit comparisons with all existing work exceed the scope of the present paper, we do, in the Appendix, spell
out the changes needed to treat the multisorted context, thus providing a framework for treating derived
Koszul duality in all the known algebraic contexts. 

It may be useful to make a brief remark on extensions of Koszul duality beyond operads: In these cases (e.g. 
PROP's), one typically discusses Koszul duality only for the PROP's themselves, but {\em not for algebras}
over them. Such algebras may have both operations and co-operations (e.g. Hopf algebras), 
which causes them to behave badly 
from a categorical point of view. For example, the forgetful functor from such algebras into chain complexes
typically is neither a left nor a right adjoint. For that reason, those types of algebras do not easily fit the Quillen
framework \cite{qc}, and what can be said about Koszul duality in this case remains largely an open problem.

The present paper is organized as follows: Because of the delicate technical nature of 
our construction, and the presence of a large number of variants of similar concepts
in the literature, we found it necessary to make the present paper as self-contained
as possible, even at the cost of redefining some
known concepts, when there is ambiguity in them. In Section \ref{prelim}, we treat these necessary technical
prerequisites. In Section \ref{st}, we treat rigorously the notion of a {\em tree}. For us, trees
are ordered, or ``planar'' trees. Because of the sign issue, which is central to us, we also
choose to be more pedantic and rigorous than is customary in this context. 
In Section \ref{sbar}, we review the Ginzburg-Kapranov bar construction in the 
setting we need, and implement the relevant sign devices. In Section \ref{sd},
we introduce our version of the concept of a D-structure. In Section \ref{sder},
we construct the derived category of D-structures, and prove that it is equivalent to
the Quillen derived category of algebras over the original operad. In the Appendix, we treat the
multi-sorted case.

\section{Preliminaries}\label{prelim}
In this paper, we will work with unital operads $\mathcal{C}$  in the category $K$-$Chain$ of chain complexes of $K$-modules
where $K$ is a field. We will also call them {\em DG $K$-module operads}.
This means a sequence $\mathcal{C}(n)$ of chain complexes of $K$-modules,
together with an action of the permutation group $\Sigma_n$ on $\mathcal{C}(n)$, $n=0,1,2,\dots$ a
unit chain map 
$$K\r \mathcal{C}(1),$$
and an operation
\beg{eoperad1}{\mathcal{C}(n_1)\otimes\dots\otimes \mathcal{C}(n_k)\otimes \mathcal{C}(k)
\r \mathcal{C}(n_1+\dots+n_k),}
$$x_1\otimes\dots\otimes x_k\otimes x_{k+1}\mapsto
\gamma(x_1,\dots,x_k;x_{k+1})$$
satisfying the usual equivariance, associativity and unit axioms modelled on the example
$\mathcal{H}_M(n)=Hom(M^{\otimes n},M)$ where $M$ is a chain complex of $K$modules,
and $Hom$ denotes the internal Hom in the category of chain complexes of $K$-modules
(\cite{operad}). As usual, when dealing with graded objects, we apply the Koszul
sign
$$(-1)^{jk}$$
in the switch homomorphism between $x\otimes y$ and $y\otimes x$ for homogeneous elements
$x,y$ of degrees $j$, $k$, respectively. 

It may be more common to put the $\mathcal{C}(k)$ term first in the
tensor product \rref{eoperad1}. We chose the current order of variables because
we work in the context of trees, which we visualize as having roots in the bottom:
When writing the trees in one line, it seems natural to write the upper
parts of the tree to the left and the root to the right. It is, obviously, only a matter
of signs.

It is also useful to introduce the operation
$$\gamma_j:\mathcal{C}(k)\otimes\mathcal{C}(n)\r \mathcal{C}(n+k-1)$$
given by
$$\gamma_j(x,y)=\gamma(\underbrace{1,\dots,1}_\text{$j-1$ times},x,
\underbrace{1,\dots,1}_\text{$n-j$ times},y).$$
When considering non-unital operads (which we do not do in this paper, but which is,
for example, the basic setup of \cite{gk}), one usually does include the operations $\gamma_j$
in the definition: the operation $\gamma$ can be recovered from them, but not vice versa.

A morphism of operads is a chain map which preserves the operations $\gamma$, the unit
and the $\Sigma_n$-equivariances. By a {\em DG-$\mathcal{C}$-algebra $A$} we mean
a homomorphism of operads 
\beg{eoperad1+}{\mathcal{C}\r\mathcal{H}_{A}.}
This is equivalent data to having operations
$$\theta:\underbrace{A\otimes\dots\otimes A}_\text{$n$ times}\otimes \mathcal{C}(n)\r A$$
which satisfy the expected associativity, unit and equivariance conditions. Morphisms of $\mathcal{C}$-algebras
are chain maps which preserve all the operations.

We will also use the notion of a monad, which is a generalization of an operad (in the sense that for every operad,
there is an associated monad, and the algebras are the same). A monad in
a category $Q$ is a functor $C:Q\r Q$ together with a product $\mu:CC\r C$ and a unit
$\eta:Id\r C$ which satisfy associativity and unit axioms. A {\em $C$-algebra}
X consists of a natural transformation $\theta:CX\r X$ which satisfies an associativity
with respect to $\mu$ and a unit axiom with respect to $\eta$. Morphisms of $C$-algebras
are morphisms in $Q$ which commute with the operation $\theta$:
$$
\diagram
CX\dto_{Cf}\rto^{\theta}&X\dto^{f}\\
CX\rto^\theta & X.
\enddiagram
$$

For an operad $\mathcal{C}$, there is a canonical monad $C$ such that $\mathcal{C}$-algebras
are the same as $C$-algebras (they form canonically equivalent categories). The
monad is given by
\beg{emonad}{CX=\bigoplus_{n\geq 0} X^{\otimes n}\otimes_{\Sigma_n} \mathcal{C}(n)}
where the symmetric group acts on $X^{\otimes n}$ by permutations (with the Koszul
signs). 

In the category of chain complexes of $K$-modules, by an {\em equivalence} we mean
a chain map which induces an isomorphism in homology. If $X$ is a chain complex, $X[n]$ denotes
the chain complex $X$ with degrees (in topology often referred to as 
{\em dimensions}) shifted by $n$: $X[n]_k=X_{k-n}$.

In important point is that for a general operad, when $f:X\r Y$ is an equivalence,
$Cf:CX\r CY$ may not be an equivalence. In fact, this has less to do with the notion
of an operad than with structure of chain complexes of $K[\Sigma_n]$-modules. 
To illustrate the problem, let $I[\Sigma_n]$ denote the chain complex of $K[\Sigma_n]$-modules
where 
$$I[\Sigma_n]_i= \begin{array}[t]{ll}
K[\Sigma_n] &\text{for $i=0,1$}\\[2ex]
0 &\text{else,}
\end{array}$$
where the differential is the identity. Then we have an obvious chain map of $K[\Sigma_n]$-modules
$\epsilon:K[\Sigma_n]\r I[\Sigma_n]$. (An ungraded module, when considered graded without
further specification, is considered to be in dimension $0$). 

Suppose we are given chain complexes $X_{(k)}$, $k\in \Z_{\geq -1}$,
where $X_{(-1)}=0$, 
sets $I_k$, $k\in \Z_{\geq 0}$, together with maps $q_k:I_k\r\Z$ and chain maps of $K[\Sigma_n]$-modules
$$f_k:\bigoplus_{i\in I_k} K[\Sigma_n][q_k(i)]\r X_{(k-1)}$$
such that $X_{(k)}$ is a pushout of the diagram
$$\diagram
\displaystyle\bigoplus_{i\in I_k}  K[\Sigma_n][q_k(i)]\rto^(.6){f_k}
\dto^{\bigoplus \epsilon[q_k(i)]} 
&
X_{(k-1)}\\
\displaystyle\bigoplus_{i\in I_k}  I[\Sigma_k][q_k(i)]. &
\enddiagram
$$
A chain complex of $K[\Sigma_n]$-modules $X$ is called {\em cell} if it is of the form
$$X\cong \text{co}\lim_{k} X_{(k)}$$
where $X_{(k)}$ are as above. (In particular, each $X_{(k)}$ is cell.)

We will call a chain complex of $K[\Sigma_n]$-modules {\em $\Sigma$-cofibrant} if it is
a retract of a cell chain complex of $K[\sigma_n]$-modules. (Note that when $K$
is a field of characteristic $0$, the assumption is automatically satisfied and hence
vacuos.)
We will call an DG $K$-module operad $\mathcal{C}$ {\em $\Sigma$-cofibrant} if each $\mathcal{C}(n)$ is an 
$\Sigma$-cofibrant
chain complex of $K[\Sigma_n]$-modules. The main point of considering these notions is
the following

\begin{proposition}
\label{psc}
Let $X$ be an $\Sigma$-cofibrant chain complex of $K[\Sigma_n]$-modules. Then the functor
$$X\otimes_{K[\Sigma_n]}?:\text{$K[\Sigma_n]$-modules}\r\text{$K[\Sigma_n]$-modules}$$
preserves equivalences. Consequently, if $\mathcal{C}$ is an $\Sigma$-cofibrant DG $K$-module
operad, then the associated monad
$$C:\text{chain complexes of $K$-modules}\r\text{chain complexes of $K$-modules}$$
preserves equivalences.
\end{proposition}


\section{Trees}
\label{st}

In this section, we will rigorously define, and describe basic operations on what is usually
referred to as ``planar trees''. Roughly speaking, they are (finite) rooted trees where the set of vertices
which are immediately above
each given vertex (i.e. are attached to it by an edge)
come with a specified linear ordering. This determines a linear ordering on the entire
set of vertices of the tree. 
In the context of the present paper, where there is need for extra sensitivity regarding signs,
we felt compelled to be perhaps more rigorous about this concept than is usually customary, and
write everything down ``algebraically'', without using pictures. 
In connection with this, we should, of course, note that mild variations in the concept are possible,
for example, the ordering could be reversed. One variation which could be considered more substantial
is that we distinguish between ``leaves'' and ``non-leaf vertices of valency $0$''. This is
because of the fact that we work with operads $\mathcal{C}$ and $\mathcal{C}$-algebras $A$
where we do not require $\mathcal{C}(0)=0$:
elements of $\mathcal{C}(0)$ are then attached to non-leaf vertices of valency $0$, while
elements of a $\mathcal{C}$-algebra $A$ are attached to leaves.

\vspace{3mm}

Denote $\bsf{n}=\{1,\dots,n\}$, $|S|$ for a (finite) set $S$ will denote the cardinality of $S$.

\begin{definition}

{\em In this paper, a {\em tree} $(n,s,L)$ is the following data: a subset $L\subseteq \bsf{n}$ 
(called the set of {\em leaves}) and map 
$$s:\bsf{n-1}\;(=\bsf{n}\smallsetminus\{n\})\r\bsf{n}\smin L$$
such that
\begin{enumerate}
\item $s(x)>x$ for all $x$

\item $x\leq y<s(x)\Rightarrow s(y)\leq s(x)$.

\end{enumerate}
For $i\in \bsf{n}$, the number $v_s(i)=|s^{-1}(i)|$ will be called the {\em valency} of $i$. Therefore,
for $i\in L$, we have $v(i)=0$.
}

\end{definition}
\vspace{3mm}

To interpret this definition in terms of usual planar trees, every $x$ is connected to $s(x)$ by an
edge; the vertex $s(x)$ is ``below'' the vertex $x$ in the planar tree. Therefore, the root is 
the greatest element, $n$.

\vspace{3mm}

\begin{lemma}
\label{l1}
Let $(n,s,L)$ be a tree, $n\notin L$ (note that $n\in L$ is only possible if
$n=0$) and
let 
$$\{k_1<k_2<\dots<k_m\}=s^{-1}(n).$$
Then we have $k_{m}=n-1$. Putting $k_0=0$, $n_i=k_i-k_{i-1}$, $i=1,\dots,m$,
$(n_i,s_i,L_i)$ are trees where
$$s_i(j)=s(j+k_{i-1})-k_{i-1},$$
$$L_i=\{x-k_{i-1}\;|\; x\in L\}\cap \bsf{n_i}.$$
\end{lemma}

\qed

We call $(n_i,s_i,L_i)_i$ the sequence of
{\em successor trees} of $(n,s,L)$. Let $(n,s,L)$ be
a tree with sequence of successor trees 
$$(n_1,s_1,L_1),\dots (n_m,s_m,L_m).$$ 
Let
$\sigma:\bsf{m}\r\bsf{m}$ be a permutation. Then there is a unique
tree $(n,s^\sigma,L^\sigma)$ with sequence of successor trees 
$$(n_{\sigma(1)},s_{\sigma(1)},L_{\sigma(1)}),\dots,
(n_{\sigma(m)},s_{\sigma(m)},L_{\sigma(m)}).$$
Let $\sim$ be the smallest equivalence relation on trees such that
\begin{enumerate}
\item
$(n,s,L)\sim (n,s^\sigma,L^\sigma)$ for any permutation $\sigma$ applicable.
\item
If $(n,s,L)$, $(n^\prime,s^\prime,L^\prime)$ have sequences of successor trees
$$(n_1,s_1,L_1),\dots (n_m,s_m,L_m),$$
$$(n_1,s_1^\prime,L_1^\prime),\dots (n_m,s_m^\prime,L_m^\prime)$$
where $(n_i,s_i,L_i)\sim (n_i,s_i^\prime,L_i^\prime)$, $i=1,\dots,m$, then
$$(n,s,L)\sim (n^\prime,s^\prime, L^\prime).$$
\end{enumerate}

\vspace{3mm}

\begin{lemma}
\label{l2}
For trees $(n,s,L)$, $(n,s^\prime,L^\prime)$, we have
$(n,s,L)\sim (n,s^\prime,L^\prime)$ if and only if there exists a permutation
$\sigma:\bsf{n}\r\bsf{n}$ such that $\sigma(L)=L^\prime$ and
$s^\prime(\sigma(i))=\sigma(s(i))$ when applicable. 
\end{lemma}

\qed

\vspace{3mm}
In the case of Lemma \ref{l2}, we call $\sigma$ the {\em intertwining permutation} $(n,s,L)\r (n,s^\prime,L^\prime)$.
Note that the permutation $\sigma$
may not be unique. Trees and intertwining permutations form a groupoid.

\vspace{3mm}
Next, we define edge contractions of trees.
Let $(n,s,L)$ be a tree, $n\notin L$ with sequence of successor trees
$$(n_1,s_1,L_1),\dots, (n_m,s_m,L_m).$$
Choose $j\in \bsf{m}$ with $n_j\notin L_j$ and let 
$$(p_1,t_1,M_1),\dots (p_q,t_q,M_q)$$
be the sequence of successor trees of $(n_j,s_j,L_j)$. Then there is a unique tree
$(n-1,s^{\circ}_{j}, L^{\circ}_{j})$ with sequence of successor trees
$$\begin{array}{l}(n_1,s_1,L_1),\dots,(n_{j-1},s_{j-1},L_{j-1}),\\
(p_1,t_1,M_1),\dots,(p_q,t_q,M_q),\\
(n_{j+1},s_{j+1},L_{j+1}),\dots,(n_m,s_m,L_m).
\end{array}$$
We define a tree $(n-1,s^\prime, L^\prime)$ inductively to be an
{\em edge contraction} of a tree $(n,s,L)$ if either 
$$(n-1,s^\prime,L^\prime)=(n-1,s^\circ_j,L^\circ_j)$$
for some $j$, or $(n,s^\prime, L^\prime)$ has sequence of successor trees
$$\begin{array}{l}(n_1,s_1,L_1),\dots,(n_{j-1},s_{j-1},L_{j-1}),\\
(n_j-1,s^\prime_{j},L^\prime_{j}),\\(n_{j+1},s_{j+1},L_{j+1}),
\dots, (n_m,s_m,L_m)
\end{array}$$
where $(n_j-1,s^\prime_{j},L^\prime_{j})$ is an edge contraction of
$(n_j,s_j,L_j)$. (The first case corresponds to contraction of an edge attached to the root,
the second case corresponds to contractions of other edges.)

\vspace{3mm}

\begin{lemma}
\label{l3}
For $j\in (\bsf{n-1})\smin L$, there exists a unique
edge contraction $(n-1,s^\prime,L^\prime)$ of the tree $(n,s,L)$
such that the map 
$$\tau=\tau_{s,j}:(\bsf{n-1})\r\bsf{n}$$
given by 
$$\tau(k)=\begin{array}[t]{ll}
k &\text{for $k<j$}\\
k+1 & \text{for $k\geq j$}
\end{array}$$
satisfies
$$
\tau(s^\prime(k))=
\begin{array}[t]{ll}
s\tau(k) & \text{if $s(k)\neq i$}\\
s(i) & \text{if $s(k)=i$}.
\end{array}
$$
Moreover, every edge contraction is obtained in this way. We call
$(n-1,s^\prime, L^\prime)$ the {\em edge
contraction of $(n,s,L)$ at $j$}.
\end{lemma}

\qed

We shall also use a particular {\em left inverse} $\rho=\rho_{s,j}$ of $\tau_{s,j}$ defined by
$$\tau\circ \rho = Id_{\bsf{n-1}},$$
$$\rho(j):=s(j).$$

\vspace{3mm}
A tree $(n,s,L)$ is called a {\em bush} if $L=\bsf{n-1}$. (In some texts, this is called a {\em corolla}.)
A tree $(n^\prime, s^\prime, L^\prime)$
is called a {\em leaf contraction} of a tree $(n,s,L)$ if there exist $1\leq i\leq j\leq n$, $j\notin L$, such that
$$s^{-1}(j)=\{i,i+1,\dots, j-1\}\subseteq L,
$$
$n^\prime=n-j+i$ and the map 
$$\tau=\tau_{s,i,j}:\bsf{n^\prime}\r\bsf{n}$$
given by 
$$\tau_{s,i,j}(k)=\begin{array}[t]{ll}
k & \text{for $k<i$}\\
k+j-i & \text{for $k\geq i$}
\end{array}
$$
satisfies
$$\tau(s^\prime(k))=s(\tau(k)) \;
\text{for all $k\in \bsf{n^\prime}$,}
$$
$$L^\prime =\tau^{-1}(L)\cup \{i\}.$$
We then call $(n^\prime,s^\prime, L^\prime)$ the {\em leaf contraction of $(n,s,L)$ at $(i,j)$}. Pictorially,
this corresponds to contracting a bush of all leaves attached to a vertex, which has no other successor trees, to a single point. Again,
we shall also use a left inverse $\rho=\rho_{s,i,j}$ of $\tau_{s,i,j}$ where
$$\rho\circ\tau=Id_{\bsf{n}^\prime},$$
$$\rho(k):=i \;\text{for $k\in\{i,i+1,\dots,j-1\}$}.$$

\vspace{3mm}

\begin{lemma}\label{l4}
Let $\sigma:\bsf{n}\r \bsf{n}$ be an intertwining permutation
$$\sigma:(n,s,L)\r (n,t,M).$$

1. Let $j\in (\bsf{n-1})\smin L$ Let $(n-1,s^\prime, L^\prime)$ resp. $(n-1,t^\prime, M^\prime)$
be the edge contractions of $(n,s,L)$ resp. $(n,t,M)$ at $j$ resp. $\sigma(j)$.
Then the unique permutation 
$$\sigma^\prime:\bsf{n-1}\r\bsf{n-1}$$
with
$$\tau_{t,\sigma(j)}\sigma^\prime=\sigma\tau_{s,j}$$
intertwines 
$$\sigma^\prime:(n-1,s^\prime, L^\prime)\r(n-1,t^\prime,M^\prime).$$

2. Let $(n^\prime, s^\prime, L^\prime)$ be a leaf contraction of $(n,s,L)$
at $(i,j)$. Then there exists a leaf contraction $(n^\prime, t^\prime,M^\prime)$
at $(\sigma(j)-j+i,\sigma(j))$ and the unique permutation 
$$\sigma^\prime:\bsf{n^\prime}\r\bsf{n^\prime}$$
which satisfies
$$\tau_{t,\sigma(j)-k+i}\sigma^\prime=\sigma\tau_{s,i,j}
$$
intertwines
$$\sigma^\prime:(n^\prime, s^\prime, L^\prime)\r (n^\prime, t^\prime, M^\prime).$$

\end{lemma}

\qed

\vspace{3mm}

\section{The augmented Ginzburg-Kapranov bar construction}
\label{sbar}

\vspace{3mm}

\noindent
{\bf An explanation of signs:} Signs are well known to cause pains throughout algebraic topology. Traditionally, one 
deals with this by writing terms in certain order, and attaching a specific sign convention when order of elements is 
exchanged. The Koszul-Milnor signs, where an exchange of elements of degree $k$ and $\ell$ introduces a sign
of $(-1)^{k\ell}$, is well known. Still, in a more complicated setting, such as constructions indexed by trees, with 
multiple operations and differentials, sign conventions become complicated and their consistency is difficult to verify.
Ginzburg and Kapranov \cite{gk} introduced a very elegant solution, making their sign an element of a 
``determinant line" on a space generated by the edges. Tensoring with this line, and interpreting operators contracting
lines as left differentiations eliminates the need to discuss signs explicitly, as the determinant line is ``functorial",
and it is therefore 
immediate that omitting two lines in opposite order always reverses signs, no matter where the corresponding 
variables are written (by the functoriality, they may as well be written, say, in the first two places). 

While the present paper was originally written with explicit signs depending on the order of elements, and we had all the 
sign consistencies checked, the signs seemed artificial and not illuminating. Additionally, it was difficult to be honestly 
sure that there is no hidden mistake. Because of that, we very much wanted to use the Ginzburg-Kapranov 
determinant line for signs attached to contractions of edges, but there was one problem: The Ginzburg-Kapranov determinant does not handle differential graded objects with their own 
internal differential, since those require their own sign scheme (such as the Koszul-Milnor signs), and we must also
express how both of the two kinds of signs are compatible. 

Because of that, we searched for an appropriate receptacle
line for the internal 
sign of a tensor product of differential graded objects: we require naturality with respect to change of order of 
tensor factors, merging two factors by a "graded-commutative" product, and anticommutation
of internal differentials on two different tensor factors. An exterior algebra 
is clearly not appropriate for this purpose if we want the internal differential to be accompanied by 
a simple operator such as multiplication by a generator corresponding to the given tensor factor (or, alternately, 
differentiation - we ended up using multiplication). While the internal differential satisfies $dd=0$ on elements, 
this is not true on degrees: One must be able to switch back and forth between even and odd degrees using the same
operator. 

We then realized that these desiderata pretty much uniquely specify the moduls of monomials in a Clifford algebra, 
where the square of the sign element corresponding to each tensor factor is $1$:
one has the correct commutation relations, while the sign of an internal differential in a given factor is expressed
by multiplication by that factor. It is then plain to see that two internal differentials of different tensor factors 
automatically anti-commute, eliminating lengthy and artificial verifications. Formally, the sign-holding objects we introduce
are as follows:

\vspace{3mm}
For a tree $(n,s,L)$ and $i\in\bsf{n}$, denote
$$\epsilon(i)=\epsilon_{(n,s,L)}(i)=|\bsf{i}\smin L|.$$
Let $K$ be a field and let $\mathcal{C}$ be an $\Sigma$-cofibrant operad in the category of chain
complexes of $K$-modules (see the definitions in Section \ref{prelim} above).
Let $A$ be a $\mathcal{C}$-algebra in the category $K$-$Chain$ (again, see Section \ref{prelim}). 
We let 
$$\Lambda(n,s,L)$$
be a $K$-valued exterior algebra on indeterminates $e_i$, $i\in \bsf{n-1}\smallsetminus L$. Let
$$C_n$$
be a Clifford algebra on indeterminates $f_i$, $i\in \bsf{n}$, by which we mean the
factor algebra of the tensor algebra on the $f_i$'s by the relations
$$f_{i}^{2}=1,\; f_if_j=-f_jf_i \;\text{for $i\neq j$}.$$
(Although the elements $e_i$, $j_j$ carry sign information, we consider them to be in degree $0$.)
We also introduce a $K$-algebra structure on
$$\Lambda(n,s,L)\otimes C_n,$$
by setting
$$f_j\cdot e_i=-e_i\cdot f_j.$$

We shall write 
$$Det(n,s,L)$$
for the top degree summand of $\Lambda(n,s,L)$, i.e. the sub-$K$-module of
$\Lambda(n,s,L)$ spanned by
$$\prod_{i\in (\bsf{n-1})\smallsetminus L} e_{i}.$$
For $\epsilon=(\epsilon_1,\dots,\epsilon_n)\in (\Z/2)^n$, we denote by $C_{n,\epsilon}$ the
sub-$K$-module of $C_{n}$ spanned by
$$f_{1}^{\epsilon_1}\dots f_{n}^{\epsilon_n}.$$
For $\varepsilon\in\Z/2$, and a graded $K$-module $M$, denote by $M_\varepsilon$
the submodule spanned by all homogeneous elements of degrees equal to $\varepsilon$ $\mod 2$.

The {\em augmented bar construction} $\widetilde{B}_{\mathcal{C}}(A)$
is defined by
\beg{esbar+}{\begin{array}{l}\widetilde{B}_\mathcal{C}(A)=\\[2ex]
\displaystyle
\Bigg(\bigoplus_{(n,s,L)}\bigoplus_{\epsilon\in (\Z/2)^n}Det(n,s,L)\otimes\\[4ex]
\displaystyle
 C_\epsilon\otimes(
\bigotimes_{i\in L} A_{\epsilon_i}
\otimes\bigotimes_{\begin{array}[t]{c}i\in\bsf{n}\smin L\\v(i)=k\end{array}}\mathcal{C}(k)_{\epsilon_i}
)[\epsilon_{(n,s,L)}(n)]\Bigg)/\sim.\end{array}}
The direct sum is over all trees. The equivalence is the smallest congruence of $K$-modules satisfying, for
$$\sigma:(n,s,L)\r (n,s^\prime, L^\prime),$$
$$\begin{array}{l}
\displaystyle
\prod_{i} e_i\prod_j f_j\cdot
\bigotimes_{i\in L} (x_i\in A)\otimes\bigotimes_{\begin{array}[t]{c}i\in\bsf{n}\smin L\\v_s(i)=k\end{array}}
(x_i\in \mathcal{C}(k))\\[8ex]
\displaystyle
\sim
\prod_i e_{\sigma(i)}\prod_j f_{\sigma(j)}\cdot\bigotimes_{i\in L^\prime} (y_i\in A)\otimes
\bigotimes_{\begin{array}[t]{c}i\in\bsf{n}\smin L^\prime\\v_{s^\prime}(i)=k\end{array}}(y_i\in \mathcal{C}(k)).
\end{array}
$$
Note: since the algebras $\Lambda(n,s,L)$, $C_n$ are not commutative, the products are taken
in the order of indexing. The congruence relation identifies the summands corresponding to equivalent
trees in the sense of Section \ref{st}, i.e. with respect to permuting successor trees of a given vertex, with
the appropriate signs (note that permuting the exterior resp. Clifford generators introduces the sign
associated with putting them back in the original order).

The differential $d$ on $\widetilde{B}_\mathcal{C}(A)$ is the sum of the following kinds of maps:
\begin{enumerate}
\item
{\bf The edge contraction summands.}
If $(n-1,s^\prime, L^\prime)$ is an edge contraction of $(n,s,L)$ at $q$, 
let 
$$\lambda_\rho:\Lambda(n,s,L)\r \Lambda(n-1,s^\prime, L^\prime)$$
be given by
$$e_i\mapsto e_{\rho(i)},$$ 
and let 
$$c_\rho:C_n\mapsto C_{n-1}$$
be defined by
$$f_j\mapsto f_{\rho(j)}.$$
Then we have
$$\begin{array}{l}
\displaystyle
\prod_{i} e_i\prod_j f_j\cdot\bigotimes_{i\in L} (x_i\in A)\otimes\bigotimes_{\begin{array}[t]{c}i\in\bsf{n}\smin L\\v_s(i)=k\end{array}}
(x_i\in \mathcal{C}(k))\\[8ex]
\displaystyle
\mapsto
\lambda_\rho(\frac{\partial}{\partial e_j}\prod_i e_i)
c_\rho(\prod_j f_j)\cdot\bigotimes_{i\in L^\prime} (y_i\in A)\otimes
\bigotimes_{\begin{array}[t]{c}i\in(\bsf{n-1})\smin L^\prime\\v_{s^\prime}(i)=\ell\end{array}}(y_i\in 
\mathcal{C}(\ell))
\end{array}
$$
where
$$y_i=x_{\tau_{s,j}(i)} \;\text{if $i\neq s(q)-1$},$$
and
$$y_{s(q)-1}=\gamma_\ell(x_q,x_{s_{q}})$$
if
$$s^{-1}(s(q))=\{q_1<\dots<q_m\},\; q_\ell=q.$$

\item
{\bf The leaf contraction summands.} If $(n^\prime, s^\prime, L^\prime)$ is a leaf
contraction of $(n,s,L)$ at $(p,r)$, 
let, again,
$$\lambda_\rho:\Lambda(n,s,L)\r \Lambda(n^\prime,s^\prime, L^\prime)$$
be given by
$$e_i\mapsto e_{\rho(i)},$$ 
and let 
$$c_\rho:C_n\mapsto C_{n^\prime}$$
be defined by
$$f_j\mapsto f_{\rho(j)}.$$
$$\begin{array}{l}
\displaystyle
\prod_{i} e_i\prod_j f_j\cdot
\bigotimes_{q\in L} (x_q\in A)\otimes\bigotimes_{\begin{array}[t]{c}q\in\bsf{n}\smin L\\v_s(q)=k\end{array}}
(x_q\in \mathcal{C}(k))\\[8ex]
\displaystyle
\mapsto
\lambda_\rho(\frac{\partial}{\partial e_j}\prod_i e_i)
c_\rho(\prod_j f_j)\cdot
\bigotimes_{q\in L^\prime} (y_q\in A)\otimes
\bigotimes_{\begin{array}[t]{c}q\in(\bsf{n^\prime})\smin L^\prime\\v_{s^\prime}(q)=\ell\end{array}}(y_q\in 
\mathcal{C}(\ell))
\end{array}
$$
where
$$y_q=x_{\tau_{s,i,j}(q)}\;\text{if $q\neq r$},$$
$$y_p=\theta(x_p,\dots,x_{r-1};x_r).$$

\item
{\bf The internal differential summands.} Denoting by $\partial$ the internal differential on $A$ and
$\mathcal{C}(k)$ for each $i\in \bsf{n}$, we have
$$\begin{array}{l}
\displaystyle
\prod_{\ell} e_\ell\prod_j f_j\cdot
\bigotimes_{q\in L} (x_q\in A)\otimes\bigotimes_{\begin{array}[t]{c}q\in\bsf{n}\smin L\\v_s(q)=k\end{array}}
(x_q\in \mathcal{C}(k))\\[8ex]
\displaystyle
\mapsto
f_i\prod_{\ell} e_\ell\prod_j f_j\cdot\bigotimes_{q\in L} (y_q\in A)\otimes
\bigotimes_{\begin{array}[t]{c}q\in(\bsf{n})\smin L\\v_{s}(q)=k\end{array}}(y_q\in 
\mathcal{C}(k))
\end{array}
$$
where
$$y_q=x_q\;\text{for $q\neq i$},$$
$$y_i=\partial(x_i).$$

\end{enumerate}

One readily verifies that the sum $d$ of all these maps $\widetilde{B}_\mathcal{C}(A)\r\widetilde{B}_\mathcal{C}(A)$
does, in fact, satisfy
$$d\circ d=0.$$ 
This, in effect, follows from the anticommutation of the operations $\partial/\partial e_i$, and left multiplication
by the generators $f_j$.

\vspace{3mm}
\begin{proposition}
\label{p1}
$$H(\widetilde{B}_{\mathcal{C}}(A),d)=0.$$
\end{proposition}

\begin{proof}
Define, for a tree $(n,s,L)$, a tree $(n+1,\overline{s},L)$ by
$$\overline{s}(i)=s(i)\;\text{for $i<n$,}$$
$$\overline{s}(n)=n+1.$$
Define a map of degree $1$
$$h:\widetilde{B}_{\mathcal{C}}(A)\r\widetilde{B}_{\mathcal{C}}(A)$$
by
$$\begin{array}{l}
\displaystyle
\prod_i e_i \cdot \prod_j f_j\cdot\bigotimes_{j\in L} (x_j\in A)\otimes\bigotimes_{\begin{array}[t]{c}j\in\bsf{n}\smin L\\v_s(j)=k\end{array}}
(x_j\in \mathcal{C}(k))\\[8ex]
\displaystyle
\mapsto
e_{n}\cdot \prod_i e_i \cdot \prod_j f_j\cdot\bigotimes_{j\in \overline{L}} (y_j\in A)\otimes
\bigotimes_{\begin{array}[t]{c}j\in(\bsf{n+1})\smin {L}\\v_{\overline{s}}(j)=k\end{array}}(y_j\in 
\mathcal{C}(k))
\end{array}
$$
where
$$y_{j}=x_j\; \text{for $j\in \bsf{n}$},$$
$$y_{n+1}=1.
$$
(Note that the tree $(n+1,\overline{s},\overline{L})$ is obtained from $(n,s,L)$ by ``grafting'',
i.e. by attaching a new root below the old root, and connecting them with an edge. The labels
of the ``grafter tree'' stay the same, the label of the new root is $n+1$.)

Then 
$$dh+hd=Id.$$
\end{proof}

\vspace{3mm}
Let
\beg{e+}{
\diagram
A\rto^(.35)\subseteq & \widetilde{B}_{\mathcal{C}}(A)\\
a \rto |<\stop &a\in \displaystyle \bigoplus_{(1,*,\{1\})} A.
\enddiagram
}
Denote
\beg{e+*}{B_\mathcal{C}(A):=(\widetilde{B}_{\mathcal{C}}(A)/A)[-1].
}
Then there is a map
\beg{e++}{\mu:B_{\mathcal{C}}(A)\r A
}
given by
$$\begin{array}{l}
\displaystyle
\bigotimes_{j\in L} (x_j\in A)\otimes\bigotimes_{\begin{array}[t]{c}j\in\bsf{n}\smin L\\v_s(j)=k\end{array}}
(x_j\in \mathcal{C}(k))\\[8ex]
\displaystyle
\mapsto
\theta(x_1,\dots,x_{n-1};x_n) \;\text{if $(n,s,L)$ is a bush}\\[2ex]
\displaystyle
\mapsto 0 \; \text{else.}
\end{array}
$$

\vspace{3mm}
\begin{proposition}
\label{p*}
(1) The map \rref{e++} is a chain map, and an equivalence for an $\Sigma$-cofibrant DG $K$-module
operad $\mathcal{C}$. 

(2)
There is a natural DG $\mathcal{C}$-algebra structure on $B_{\mathcal{C}}(A)$ such that
\rref{e++} is a map of DG $\mathcal{C}$-algebras.

\end{proposition}

\begin{proof}
For (1), to prove that \rref{e++} is a chain map, it suffices to prove that it vanishes on differentials
of trees $(n,s,L)$ where $\epsilon_{(n,s,L)}(n)=2$. On such trees, however, the differential has
two summands (one edge contraction and one leaf contraction), which cancel after applying
\rref{e++} to them. 

By definition, further, \rref{e++} fits into the following diagram of chain complexes:
\beg{ep*i}{
\diagram
A[-1]\dto^{Id}\rto & \widetilde{B}_{\mathcal{C}}(A)[-1]\dto^{\widetilde{\mu}}\rto 
& B_\mathcal{C}(A)\dto^\mu\\
A[-1]\rto & A\otimes I[-1]\rto & A
\enddiagram
}
where $I$ is the chain complex of $K$-modules
$$\diagram
K\rto^\cong & K
\enddiagram
$$
in dimensions $1,0$. 

The map $\widetilde{\mu}$ has
$$a\in \bigoplus_{(1,*,\{1\})} A\mapsto a.$$
Since the source and target of $\widetilde{\mu}$ are both acyclic,
$\widetilde{\mu}$ is an equivalence, so
$\mu$ is an equivalence by \rref{ep*i} and the $5$-lemma.

The $\mathcal{C}$-algebra structure on $B_\mathcal{C}(A)$ is given by
$$\begin{array}{l}
\displaystyle
\prod_i e_{i,q}\prod_{j<n_q} f_{j,q}\cdot \prod_{q=1}^m f{n_q,q}\cdot\\[2ex]
\displaystyle
\bigotimes_{q=1}^m
\left(
\bigotimes_{j\in L_q} (x_{j,q}\in A)\otimes\bigotimes_{\begin{array}[t]{c}j\in\bsf{n_q}\smin L_q\\v_s(j)=k\end{array}}
(x_{j,q}\in \mathcal{C}(k))\right)\otimes (x\in \mathcal{C}(m))
\\[8ex]
\displaystyle
\mapsto
\prod_q  (\prod_i e_{\ell_q(i)}\cdot \prod_{j<n_q} f_{\ell_q(j)}\cdot f_n\cdot
\bigotimes_{j\in {L}} (y_j\in A)\otimes
\bigotimes_{\begin{array}[t]{c}j\in\bsf{n}\smin L\\v_{{s}}(j)=k\end{array}}(y_j\in 
\mathcal{C}(k))
\end{array}
$$
where the left had side indicates a typical element of 
$$(B_\mathcal{C}(A))^{\otimes m}\otimes \mathcal{C}(m),$$
with $e_{i,q}$, $f_{j,q}$ denoting the exterior and Clifford elements assigned to the $q$'th $B_\mathcal{C}(A)$ 
factor, $(n_q,s_q,L_q)$ are successor trees of a tree $(n,s,L)$, and we write
$$\ell_q(i)=i+(n_1-1)+\dots+(n_{q-1}-1),$$
$$n=\sum_{q=1}^{m} n_q-m+1,$$
for $s_q(j)\neq n_q$,
$$
s(\ell_q(j))=\ell_q(s_q(j)),$$
for $s_q(j)=n_q$,
$$s(\ell_q(j))=n,$$
for $1\leq j<n_q$,
$$y_{\ell_1(j)}=x_{j,q},$$
$$y_n=\gamma(x_{n,1},\dots,x_{n,q};x).$$
\end{proof}

\vspace{3mm}
\section{D-structures}
\label{sd}

\begin{proposition}
\label{p**}
There exists a graded homomorphism of $K$-modules
$$\bigoplus_{n\geq 0} \delta_n=\delta:\widetilde{B}_{\mathcal{C}}(A)\r
\bigoplus_{n\geq 0} \widetilde{B}_{\mathcal{C}}(A)^{\otimes n}\otimes_{\Sigma_n}\mathcal{C}(n)=B_\mathcal{C}(A)$$
such that if we denote, for $x\in \widetilde{B}_\mathcal{C}(A)$,
$$\delta(x)=\sum (\delta^{\prime}_n(x))_1\otimes\dots\otimes (\delta_n^\prime(x))_n\otimes \delta_n^{\prime
\prime}(x))$$
with $(\delta_{n}^\prime(x))_i\in \widetilde{B}_{\mathcal{C}}(A)$, $\delta_n^{\prime\prime}(x)\in \mathcal{C}(n)$,
then
$$\begin{array}{l}
d(x_1\otimes\dots\otimes x_n\otimes x_{n+1})\\[2ex]
=\displaystyle \sum_{i=1}^{n+1}(-1)^{|x_1|+\dots |x_{i-1}|}
x_1\otimes\dots\otimes x_{i-1} \otimes d x_i \otimes x_{i+1}
\otimes \dots \otimes x_{n+1}\\[4ex]
\displaystyle +\sum_{i=1}^{n}\sum_{m\geq 0}
(-1)^{|x_1|+\dots+|x_i|+|\delta_m^{\prime\prime}(x_i)|\cdot (|x_{i+1}|+\dots +|x_n|)}x_1\otimes\dots\\[4ex]
\dots x_i\otimes (\delta_m^\prime(x_i))_1\otimes\dots\otimes (\delta^\prime_m(x_i))_m
\otimes x_{i+1}\otimes \dots\\[2ex]
\dots \otimes x_n\otimes \gamma_i(\delta^{\prime\prime}_n(x_i),x_{n+1}).
\end{array}$$
\end{proposition}

\begin{proof}
For an element $x\in\widetilde{B}_\mathcal{C}(A)$ indexed over a tree of the form $(1,*,\{1\})$
(thus, $x\in A$), we put
$$\delta(x)=0.$$
For elements indexed over a tree $(n,s,L)$ with $n\notin L$ with
successor trees $(n_i,s_i,L_i)$, $i=1,\dots, k$ (note that the case $n=1$ may occur, in
which case we have $k=0$), define
$$\begin{array}{l}
\displaystyle
\delta\left(\prod_\ell e_\ell\cdot \prod_j f_j\cdot\bigotimes_{j\in L} (x_j\in A)\otimes\bigotimes_{\begin{array}[t]{c}j\in\bsf{n}\smin L\\v_s(j)=k\end{array}}
(x_j\in \mathcal{C}(k))\right)\\[10ex]
\displaystyle
=\prod_\ell e_\ell\cdot \prod_j f_j\cdot\bigotimes_{i=1}^{k}\left(
\bigotimes_{j=n_1+\dots+n_{i-1}+1}^{n_1+\dots+n_i}x_j
\right)\otimes (x_n\in \mathcal{C}(k)).
\end{array}
$$
Note carefully that this formula hides a sign coming from the shuffle of the $e_\ell$'s and $f_j$'s so that the
variables corresponding to the individual $\widetilde{B}_{\mathcal{C}}(A)$ factors on the right
hand side are moved together.
The sign in the second summand of the differential comes from the fact that the root of the
trees corresponding to the $i$'th factor must be moved to the right to apply the operad operation. 
\end{proof}

\vspace{3mm}
\begin{definition}
\label{d2}
{\em Let $(\mathcal{C},\partial)$ be a unital $\Sigma$-cofibrant DG-operad over a field $K$. Define
$$CN=\bigoplus_{n\geq 0} N^{\otimes n}\otimes_{\Sigma_n}\mathcal{C}(n).$$
A {\em D-structure}
with respect to $\mathcal{C}$ is a DG-$K$-module $(N,d)$ together with a graded homomorphism
of $K$-modules
$$\delta:N\r CN,$$
$$x\mapsto\sum_ n ((\delta^{\prime}_n(x))_1\otimes\dots\otimes (\delta_n^\prime(x))_n\otimes \delta_n^{\prime
\prime}(x)),$$
such that 
the map $\Delta:CN\r CN$ given by
\beg{ed2*}{
\begin{array}{l}
\Delta(x_1\otimes\dots\otimes x_n\otimes x_{n+1})\\[2ex]
=\displaystyle \sum_{i=1}^{n+1}(-1)^{|x_1|+\dots |x_{i-1}|}
x_1\otimes\dots\otimes x_{i-1} \otimes d x_i \otimes x_{i+1}
\otimes \dots \otimes x_{n+1}\\[4ex]
\displaystyle +\sum_{i=1}^{n}\sum_{m\geq 0}
(-1)^{|x_1|+\dots+|x_i|+|\delta_m^{\prime\prime}(x_i)|\cdot (|x_{i+1}|+\dots +|x_n|)}x_1\otimes\dots\\[4ex]
\dots x_i\otimes (\delta_m^\prime(x_i))_1\otimes\dots\otimes (\delta^\prime_m(x_i))_m
\otimes x_{i+1}\otimes \dots\\[2ex]
\dots \otimes x_n\otimes \gamma_i(\delta^{\prime\prime}_n(x_i),x_{n+1}).
\end{array}
}
defines a differential on $CN$ (i.e., $\Delta\circ\Delta=0$), and furthermore, $(CN,\Delta)$,
with the $\mathcal{C}$-algebra structure on $CN$ in the category of graded $K$-modules
coming from the monad,
is a DG-$\mathcal{C}$-algebra.
}
\end{definition}

\vspace{3mm}
For our purposes, the concept of a D-structure is motivated by the following Proposition. In the case of 
the operad defining $A$-modules over an algebra $A$, it was also used in \cite{o}. For further motivation,
see Comment 2 before Definition 13 below.

\begin{proposition}
\label{pbc}
$(\widetilde{B}_{\mathcal{C}}(A),\delta)$ is a $\mathcal{C}$-D-structure.
\end{proposition}

\begin{proof}
This follows immediately from Proposition \ref{p**}.
\end{proof}

\vspace{3mm}

\noindent
{\bf Comment:} Denote, for a $\mathcal{C}$-D-structure $(N,\Delta)$, by $C_\Delta(N)$
the DG-$\mathcal{C}$-algebra $(CN,\Delta)$. Note that the explicit notation is justified, as 
$CN$ has its own $\mathcal{C}$-algebra structure coming from the fact that $C$ is the
monad defining $\mathcal{C}$-algebras. Recall also that the unit $\eta:N\r CN$ is
a chain map.

The map $\eta:N\r C_\Delta(N)$, on the other hand, is not a chain map. In fact, it is
easily seen that the definition implies
\beg{ed2+}{
\eta d+\delta=\Delta\eta,
}
which in turn implies that
$$x\mapsto (-1)^{|x|}\delta(x)$$
is a chain map $\delta^\prime:N\r C_\Delta(N)$. Further, this map is null-homotopic by
\rref{ed2+}. In the case of $N=\widetilde{B}_\mathcal{C}(A)$, $\delta^\prime$ is
the projection 
$$\widetilde{B}_{\mathcal{C}}(A)\r B_\mathcal{C}(A)[1]$$
of \rref{e+*}. In view of Proposition \ref{p1}, then, one may ask if every D-structure has
$0$ homology. We will see that this is not the case. (By Comment 2 after Definition \ref{d3}, the 
differential $d$ on $N$ can be $0$.)

\vspace{3mm}
\begin{definition}
\label{d3}
{\em 
A {\em morphism of $\mathcal{C}$-D-structures} $f:(N,\Delta)\r (N^\prime,\Delta^\prime)$ is a homomorphism
of graded $K$-modules
$$f_0:N\r C_{\Delta^\prime}(N^\prime)$$
which extends to a homomorphism of DG-$\mathcal{C}$-algebras $\overline{f}:C_\Delta(N)
\r C_{\Delta^\prime}(N^\prime)$ in the sense of the following diagram:
$$\diagram
N\dto_\eta\drto^{f_0} &\\
C_\Delta(N)\rdotted|>\tip_{\overline{f}} &C_{\Delta^\prime}(N^\prime).
\enddiagram
$$
The {\em identity morphism} is $\eta$. For $g:(N^\prime,\Delta^\prime)\r (N^{\prime\prime},
\Delta^{\prime\prime})$, composition is defined by
$$(g\circ f)_0=\overline{g}\overline{f}\eta.$$
}
\end{definition}

\vspace{3mm}
\noindent
{\bf Comments:}
1. Note that $\overline{f}$ is uniquely determined as a map of $\mathcal{C}$-algebras because
$C_\Delta(N)$ is equal to $CN$ as a graded $K$-module and $\eta:N\r CN$ is a universal
homomorphism of graded $K$-modules into $\mathcal{C}$-algebras in the category of
graded $K$-modules.  Thus, the condition on $f_0$ reduces to requiring that $\overline{f}$
be a chain map. This can be written explicitly as follows: For 
$$x_1\otimes\dots \otimes x_n\otimes x_{n+1}\in N^{\otimes n}\otimes\mathcal{C}(n),$$
one requires that
$$
\begin{array}{l}
\displaystyle \sum_{i=1}^{n}(-1)^{|x_1|+\dots |x_{i-1}|}
f_0(x_1)\otimes\dots\otimes f_0(x_{i-1}) \otimes df_0( x_i) \otimes f_0(x_{i+1})
\otimes \dots \otimes f_0( x_{n})\otimes x_{n+1}
\\[4ex]
\displaystyle +\sum_{i=1}^{n}\sum_{m\geq 0}
(-1)^{|x_1|+\dots+|x_i|+|\delta_m^{\prime\prime}(x_i)|\cdot (|x_{i+1}|+\dots +|x_n|)}f_0(x_1)\otimes\dots\\[4ex]
\dots f_0(x_i)\otimes (\delta_m^\prime(f_0(x_i)))_1\otimes\dots\otimes (\delta^\prime_m(f_0(x_i)))_m
\otimes f_0(x_{i+1})\otimes \dots\\[2ex]
\dots \otimes f_0(x_n)\otimes \gamma_i(\delta^{\prime\prime}_n(f_0(x_i)),x_{n+1})\\[2ex]
=\displaystyle \sum_{i=1}^{n}(-1)^{|x_1|+\dots |x_{i-1}|}
f_0(x_1)\otimes\dots\otimes f_0(x_{i-1}) \otimes f_0( d x_i) \otimes f_0(x_{i+1})
\otimes \dots \otimes f_0(x_n) \otimes x_{n+1}\\[4ex]
\displaystyle +\sum_{i=1}^{n}\sum_{m\geq 0}
(-1)^{|x_1|+\dots+|x_i|+|\delta_m^{\prime\prime}(x_i)|\cdot (|x_{i+1}|+\dots +|x_n|)}f_0(x_1)\otimes\dots\\[4ex]
\dots f_0(x_i)\otimes f_0((\delta_m^\prime(x_i))_1)\otimes\dots\otimes f_0((\delta^\prime_m(x_i))_m)
\otimes f_0(x_{i+1})\otimes \dots\\[2ex]
\dots \otimes f_0(x_n)\otimes \gamma_i(\delta^{\prime\prime}_n(x_i),x_{n+1}).
\end{array}
$$

2. It follows immediately from \rref{ed2*} and \rref{ed2+} that the category of $\mathcal{C}$-$D$-structures
is in fact equivalent to the full subcategory of the category of $\mathcal{C}$-algebras on objects which
are free $\mathcal{C}$-algebras after forgetting differentials. Such objects are called {\em quasi-free} or
{\em almost free} algebras and were introducted, in the case of
augmented unital operads, by Getzler and Jones
\cite{gjones2} and further investigated by Fresse \cite{fresse}. Using this, one can deduce
the augmented case of the results of 
the present paper from \cite{gjones2,fresse}. The authors do not know if this was historically 
recognized as a restatement of non-unital (or, equivalently, augmented unital)
Koszul duality. Definition \ref{d2} generalizes the form 
in which $D$-structures first arose in geometry \cite{o}, where it was also first observed that, at least in that
special case, augmentation is not needed to make the arguments, and that this is
unital Koszul duality.

\vspace{3mm}

\begin{definition}
\label{d3}
{\em A morphism $f:(N,\Delta)\r (N^\prime,\Delta^\prime)$ of $\mathcal{C}$-D-structures is an {\em equivalence}
if the homomorphism $\overline{f}:C_\Delta(N)\r C_{\Delta^\prime}(N^\prime)$ induces an
isomorphism in homology.
}
\end{definition}

\vspace{3mm}

\begin{proposition}
\label{p+}
Let $\mathcal{C}$ be an $\Sigma$-cofibrant DG $K$-module operad.
The following functors preserve equivalences:
\beg{ep*1}{\widetilde{B}_{\mathcal{C}}(?):\text{DG-$\mathcal{C}$-algebras}\r\text{$\mathcal{C}$-D-structures}}
\beg{ep*2}{C_\Delta(?):\text{$\mathcal{C}$-D-structures}\r\text{DG-$\mathcal{C}$-algebras.}}
There exist natural equivalences
\beg{ep*3}{C_\Delta(\widetilde{B}_\mathcal{C}(?))\r Id:\text{DG-$\mathcal{C}$-algebras}
\r \text{DG-$\mathcal{C}$-algebras}
}
\beg{ep*4}{\widetilde{B}_\mathcal{C}(C_\Delta(?))\r Id:
\text{$\mathcal{C}$-D-structures}
\r \text{$\mathcal{C}$-D-structures}.
}
\end{proposition}

\begin{proof}
We have $C_\Delta\widetilde{B}_{\mathcal{C}}(A)=B_\mathcal{C}(A)$ by Proposition \ref{p**},
so the natural equivalence of DG-$\mathcal{C}$-algebras \rref{ep*3} is established by 
Proposition \ref{p*}. Since equivalences of $\mathcal{C}$-D-structures
$f:(N,\Delta)\r(N^\prime,\Delta^\prime)$ are defined as equivalences of the corresponding
DG-$\mathcal{C}$-algebras, this also implies that \rref{ep*1} preserves equivalences.

The functor \rref{ep*2} on morphisms is defined by $f\mapsto \overline{f}$. It preserves
equivalences by definition. 

To construct the natural equivalence $\epsilon$ of \rref{ep*4}, again, we need to construct
$\overline{\epsilon}$, i.e. a natural equivalence of DG-$\mathcal{C}$-algebras
$$B_\mathcal{C}(C_\Delta(?))\r C_{\Delta}(?),$$
which is the natural equivalence of Proposition \ref{p*} applied to $C_\Delta$.
\end{proof}

\vspace{3mm}

\section{Derived categories}
\label{sder}

The material covered in this section is well known. We cover it for the lack of convenient
reference.

\begin{definition}
\label{ddd}
Let $C$ be a category, and let $\mathcal{E}\subseteq Mor(C)$ be an
arbitrary class of morphism, which we call {\em equivalences}. A {\em derived category}
$DC$ is a category together with a functor 
$$\Phi:C\r DC$$
such that 
$$\phi\in \mathcal{E}\Rightarrow \text{$\Phi(\phi)$ is an isomorphism},$$
and for every functor $F:C\r B$ such that 
$$\phi\in \mathcal{E}\Rightarrow \text{$F(\phi)$ is an isomorphism},$$
there exists a functor $DF:DC\r D$ and a natural isomorphism
$$\diagram \eta:F\rto^\cong & DF\circ\Phi,
\enddiagram
$$
which is further unique in the following sense: For any other functor
$DF^\prime:DC\r D$ and a natural isomorphism
$$\diagram \eta^\prime:F\rto^\cong & DF^\prime\circ\Phi,
\enddiagram
$$
there exists a unique natural isomorphism 
$$\diagram \xi:DF\rto^\cong & DF^\prime
\enddiagram$$
such that
\beg{eddd+}{
\xi(\Phi)=\eta^\prime\circ \eta^{-1}.
}
\end{definition}

\vspace{3mm}
Note that the existence of a derived category is not automatic, and in particular cannot be
proved by the usual algebraic ``localization'' argument because $\mathcal{E}$ is only a class,
not necessarily a set. In certain cases a derived category is known to exist, for example
when $\mathcal{E}$ is the class of equivalences in a Quillen model structure \cite{q}.

\vspace{3mm}
\begin{proposition}
\label{puni}
If a derived category exists then it is unique in the sense that for two functors $\Phi:C\r DC$,
$\Phi^\prime:C\r D^\prime C$ both satisfying the condition of Definition \ref{ddd},
there exists a natural equivalence 
\beg{epuni1}{A:DC\r D^\prime C,\; B:D^\prime C\r DC,
}
\beg{epuni2}{\diagram \epsilon:Id \rto^\cong & BA,
\enddiagram\;
\diagram \zeta:Id \rto^\cong & AB
\enddiagram
}
and natural isomorphisms
\beg{epuni3}{\diagram \eta:\Phi^\prime\rto^\cong &A\Phi,
\enddiagram\;
\diagram \kappa:\Phi \rto^\cong & B\circ \Phi^\prime
\enddiagram
}
such that
\beg{epuni4}{(B\eta)\circ\kappa=\epsilon\Phi,\; (A\kappa)\circ \eta=\zeta\Phi^\prime.
}
Furthermore, these data are unique in the following sense: For any other choice of the
data \rref{epuni1}, \rref{epuni2} satisfying \rref{epuni3}, \rref{epuni4} (which we will decorate
with a $(?)^\circ$ to make a distinction), there are unique isomorphisms
$$\diagram \alpha:A\rto^\cong &A^\circ,
\enddiagram\;
\diagram \beta:B\rto^\cong & B\circ B^\circ
\enddiagram
$$
such that
$$Id=(\zeta^\circ)^{-1}\circ(\beta A^\circ)\circ (B\alpha)\circ\zeta,$$
$$Id=(\epsilon^\circ)^{-1}\circ (\alpha B^\circ)\circ (A\beta)\circ \epsilon,$$
$$\eta^\circ=\alpha \Phi\circ\eta,\; \kappa^\circ=\beta\Phi^\prime \circ \kappa.$$
\end{proposition}

\begin{proof}
One obtains the transformations $\eta,\kappa$ as special cases of the $\eta$ in
Definition \ref{ddd}, and the $\epsilon, \zeta,\alpha, \beta$ as special cases of the
$\xi$ of Definition \ref{ddd}. The diagrams are then special cases of \rref{eddd+} and
the uniqueness of $\xi$ in Definition \ref{ddd}.
\end{proof}

\vspace{3mm}
\begin{lemma}
\label{lx}
Suppose $C_1$, $C_2$ are categories with classes of equivalences $\mathcal{E}_1$,
$\mathcal{E}_2$. Suppose that

(1) $(C_1,\mathcal{E}_1)$ has a derived category $\Phi_1:C_1\r DC_1$.

(2) There exists a pair of functors $F:C_1\r C_2$, $G:C_2\r C_1$ which preserve equivalences,
and natural equivalences 
$$\diagram \epsilon:FG\rto^\sim &Id,
\enddiagram\;
\diagram \zeta:GF\rto^\sim & Id.
\enddiagram
$$
Then $C_2$ has a derived category $\Phi_2:C_2\r DC_2$, and there exist functors
$$DF:DC_1\r DC_2,\; DG:DC_2\r DC_1$$
and natural isomorphisms
$$\diagram D\epsilon:DFDG\rto^(.6)\cong &Id,
\enddiagram\;
\diagram D\zeta:DGDF\rto^(.6)\cong &Id,
\enddiagram
$$
$$\diagram \kappa:\Phi_2F\rto^\cong &DF\Phi,
\enddiagram\;
\diagram\lambda:\Phi_1G\rto^\cong & DG\Phi_2
\enddiagram
$$
satisfying the identities
$$Id=D\zeta\circ(DG\kappa)\circ(\lambda F)\circ \epsilon^{-1},
$$
$$Id=D\epsilon \circ (DF\lambda)\circ (\kappa F)\circ \zeta^{-1}.$$
\end{lemma}

\vspace{3mm}
\noindent
{\bf Comment:} We say that $F,G$ {\em induce an equivalence of derived categories}. There is also
a uniqueness statement which we omit.

\begin{proof}
Put $DC_2:=DC_1$, $\Phi_2:=\Phi_1 G$. 
We may then define $DF=DG:=Id$, and $\lambda:=Id$. To define
$\kappa:\Phi_1GF\r \Phi_1$, put $\kappa:=\Phi_1\zeta$. The natural isomorphisms
$D\zeta$, $D\epsilon$ are then defined as special cases of the $\xi$ of Definition \ref{ddd}.
\end{proof}

\vspace{3mm}
\noindent
{\bf Comment:}
The fact that we defined $DC_2:=DC_1$ does not affect the force of this Lemma, since
a derived category is only defined up to equivalence; the force of the statement is in
the functors and natural isomorphisms involved.

\vspace{5mm}
The following is the main result of this section.
\vspace{3mm}

\vspace{3mm}
\begin{theorem}
\label{t1}
Let $\mathcal{C}$ be an $\Sigma$-cofibrant DG $K$-module operad.
There exists a derived category of $\mathcal{C}$-D-structures, and the functors of Proposition
\ref{p+} induce an equivalence of derived categories.
\end{theorem}

\begin{proof}
An immediate consequence of Proposition \ref{p+} and Lemma \ref{lx}; the derived category
of DG $\mathcal{C}$-algebras follows from the standard closed model structure (where cofibrations
are retracts of relative cell objects).
\end{proof}

\vspace{3mm}

\section{Appendix: The multi-sorted case}

In this apppendix, we will briefly explain what has to be changed in the definitions and theorems of
this paper to cover the case of multi-sorted operads (equivalently, multicategories, cf. Elmendorf, Mandell \cite{em}).

We are given a set $S$ (in the preceding sections, we assumed $S$ is a $1$-element set). An
{\em $S$-sorted operad} $\mathcal{C}$ over $K$-$Chain$ is defined in the same way as an operad over $K$-$Chain$, 
with $\mathcal(n)$
replaced by a system $\mathcal{C}(n)_{s_1,\dots,s_n;t}$ where $s_1,\dots, s_n,t$ range independently 
over all elements of $S$. For every $s\in S$, there is a unit
$$K\r \mathcal{C}(1)_{s;s}.$$
The operations \rref{eoperad1} become
$$\begin{array}{l}\displaystyle
\mathcal{C}(n_1)_{s_{1,1},\dots,s_{1,n_1};t_1}\otimes \dots
\otimes \mathcal{C}(n_k)_{s_{k,1},\dots,s_{k,n_k};t_1}\otimes
\mathcal {C}(k)_{t_1,\dots,t_k;u}\\[2ex]\displaystyle
\r \mathcal{C}(n_1+\dots +n_k)_{s_{1,1},\dots,s_{k,n_k};u}.
\end{array}$$
Similarly,
$$\gamma_j:\mathcal{C}(k)_{s_1,\dots,s_k;t_j}
\otimes \mathcal{C}(n)_{t_1,\dots,t_n;u}
\r \mathcal{C}(n+k-1)_{t_1,\dots,t_{j-1},s_1,\dots,s_k,
t_{j+1},\dots t_n;u}.$$
For an object $M$ of $(K$-$Chain)^S$ (the $S$-fold product
of categories), the $S$-sorted operad $\mathcal{H}_M$ is defined by
$$\mathcal{H}_M(n)_{s_1,\dots,s_n;t}=Hom(M_{s_1}\otimes\dots\otimes M_{s_n};M_t)$$
and for an $S$-sorted operad $\mathcal{C}$, a morphism of $S$-sorted operads \rref{eoperad1+} defines
an {\em $S$-sorted DG-$\mathcal{C}$-algebra} $M$. 

Accordingly, the monad $C$ in $(K$-$chain)^S$ corresponding to $\mathcal{C}$ is defined by
$$(CX)_t=\bigoplus_{n\geq 0}\bigoplus_{s_1,\dots,s_n\in S}
(X_{s_1}\otimes \dots \otimes X_{s_n})\otimes_{\Sigma_n} \mathcal{C}(n)_{s_1,\dots,s_n;t}.
$$
$\Sigma$-cofibrancy on an $S$-sorted operad is defined to mean that
$$\bigoplus_{(s_1,\dots,s_n)\in \mathcal{O}}\mathcal{C}(n)_{s_1,\dots,s_n;t}$$
be a $K[\Sigma_n]$-cell chain complex for every $t\in S$ and every $\Sigma_n$-orbit $\mathcal{O}$
of the product $S^n$.

Regarding trees, an {\em $S$-sorted tree $(n,s,L,f)$} is a tree $(n,S,L)$ with a map
$$f:\bsf{n}\r S.$$
An intertwining operation $\sigma:\bsf{n}\r\bsf{n}$ between $S$-sorted trees $(n,s,L,f)$
and $(n^\prime,s^\prime, L^\prime, f^\prime)$ must satisfy $f^\prime \sigma=f$. Similarly, $S$-sorted successor trees of
$(n,s,L,f)$ are $(n_i,s_i,L_i,f_i)$ where
$$f_i(j)=f(j+k_{i-1}).$$
All other definitions and statements regarding trees remain completely analogous for $S$-sorted trees.

Regarding the augmented Ginzburg-Kapranov bar construction, the definition \rref{esbar+} is modified by
replacing
$$\bigoplus_{(n,s,L)}$$
by 
$$\bigoplus_{(n,s,L,f)},$$
$\mathcal{C}(k)$ by $\mathcal{C}(k)_{f(i_1),\dots,f(i_k);f(i)}$ where
$$s(i_1)=\dots =s(i_k)=i$$
and $A$ by $A_{f(i)}$.

Thereby, $\widetilde{B}_{\mathcal{C}}(A)$ is an object of $(K$-$Chain)^S$ ($\widetilde{B}_{\mathcal{C}}(A)_s$
is the summand of the modified \rref{esbar+} over $f(n)=s$).

Propositions \ref{p1}, \ref{p*} translate readily to $(K$-$Chain)^S$. $B_\mathcal{C}(A)$ becomes an
$S$-sorted DG-$\mathcal{C}$-algebra. 

To define an $S$-sorted $D$-structure (Definition \ref{d2}), $\delta$ is an $S$-tuple of morphisms of 
graded $K$-modules, $\Delta$ is required to give $CN$ a structure of an $S$-sorted $\mathcal{C}$-algebra.
In Definition \ref{d3}, $f_0$ is an $S$-tuple of morphisms of graded $K$-modules, 
$\overline{f}$ is a morphism of $S$-sorted DG-$\mathcal{C}$-algebras. 

Again, in the spirit of \cite{gjones2}, the category of $S$-sorted $\mathcal{C}$-$D$-structures 
is equivalent to the full subcategory of the category of $\mathcal{C}$-algebras which are free $C$-algebras
after forgetting differentials.

Definition \ref{d3}, Proposition \ref{p+} and Theorem \ref{t1} remain precisely analogous, replacing 
DG-$\mathcal{C}$-algebras by $S$-sorted DG-$\mathcal{C}$-algebras.

\end{document}